\documentclass{amsart}

\usepackage{a4wide,latexsym,amssymb,amsthm,amsmath,color}
\usepackage{appendix}
\usepackage[colorlinks]{hyperref}

\theoremstyle{plain}

\newtheorem{theorem}{Theorem}
\newtheorem{lemma}{Lemma}
\newtheorem{proposition}{Proposition}
\newtheorem{corollary}{Corollary}

\theoremstyle{definition}
\newtheorem{definition}{Definition}

\theoremstyle{remark}

\newtheorem{remark}{Remark}

\def \M{\mathcal{M}}
\def \S{\mathcal{S}}
\def \g{\bar{g}}
\def \tangsp{\mathfrak{X}(\S)}
\def \normsp{\mathfrak{X}(\S)^\perp}

\def \hodge{\star^\perp}
\def \tr{\textsf{tr}}

\def \tildeh{\widetilde{h}}
\def \tildeA{\widetilde{A}}
\def \id{\mathbf{1}}

\def \B{\mathcal{B}}
\def \J{\mathcal{J}}
\def \T{\mathcal{T}(\S)}

\newcommand{\scprod}[2]{\langle{#1, #2}\rangle}

\newcounter{mnotecount}

\newcommand{\mnotex}[1]
{\protect{\stepcounter{mnotecount}}$^{\mbox{\footnotesize $\bullet$\themnotecount}}$ 
\marginpar{
\raggedright\tiny\em
$\!\!\!\!\!\!\,\bullet$\themnotecount: #1} }

\begin{document}

\title[Umbilical properties of spacelike co-dimension two submanifolds]
{Umbilical properties of spacelike \\ co-dimension two submanifolds}

\author[N. Cipriani]{Nastassja Cipriani}
\address{KU Leuven, Department of Mathematics, Celestijnenlaan 200B -- Box 2400, BE-3001 Leuven, Belgium and F\'isica Te\'orica, Universidad del Pa\'is Vasco, Apartado 644, 48080 Bilbao, Spain} 
\email{nastassja.cipriani@wis.kuleuven.be}

\author[J.M.M. Senovilla]{Jos\'e M. M. Senovilla}
\address{F\'isica Te\'orica, Universidad del Pa\'is Vasco, Apartado 644, 48080 Bilbao, Spain} \email{josemm.senovilla@ehu.es}

\author[J. Van der Veken]{Joeri Van der Veken}
\address{KU Leuven, Department of Mathematics, Celestijnenlaan 200B -- Box 2400, BE-3001 Leuven, Belgium} 
\email{joeri.vanderveken@wis.kuleuven.be}

\thanks{This work is partially supported by the Belgian Interuniversity Attraction Pole P07/18 (Dygest) and was initiated during a visit of the first author to University of the Basque Country supported by a travel grant of the Research Foundation -- Flanders (FWO). NC and JMMS are supported under grant FIS2014-57956-P (Spanish MINECO--Fondos FEDER). JMMS is also supported under project UFI 11/55 (UPV/EHU)}

\begin{abstract}
For Riemannian submanifolds of a semi-Riemannian manifold, we introduce the concepts of \emph{total shear tensor} and \emph{shear operators}
as the trace-free part of the corresponding second fundamental form and shape operators. The relationship between these quantities and the umbilical properties of the submanifold is shown.
Several novel notions of umbilical submanifolds are then considered along with the classical concepts of totally umbilical and pseudo-umbilical submanifolds.

Then we focus on the case of co-dimension $2$, and we present necessary and sufficient conditions for the submanifold to be umbilical with respect to a normal direction. Moreover, we prove that the umbilical direction, if it exists, is unique ---unless the submanifold is totally umbilical--- and we give a formula to compute it explicitly. When the ambient manifold is Lorentzian we also provide a way of determining its causal character. We end the paper by illustrating our results on the Lorentzian geometry of the Kerr black hole.

\end{abstract}



\maketitle


\section{Introduction}
Co-dimension two spacelike submanifolds play a distinctive and central role in gravitational theories based on Lorentzian geometry, especially in the prominent theory of general relativity. From a mathematical point of view this is due to the fact that the {\em extrinsic} properties of such submanifolds, described by their shape operators, encode their local, infinitesimal variation along normal directions, and in particular along causal (timelike or null) directions. In plain words, they give instantaneous information about their {\em evolution}. 

The so-called trapped, or marginally trapped, submanifolds \cite{GlobalLorentzianGeometry,Kriele,ONeill} are an outstanding example of the importance of such submanifolds for the case of co-dimension 2. They are characterized by the causal orientation of their mean curvature vector field: future timelike for trapped surfaces, and future null for marginally trapped ones. They represent situations with strong gravity, and arise in deep mathematical results such as the singularity theorems \cite{GlobalLorentzianGeometry,HawkingEllis,Kriele,SenovillaMilestone} or in the study of the several types of horizons enclosing black holes \cite{HawkingEllis,Kriele,SenovillaMilestone}. These horizons are co-dimension one submanifolds (hypersurfaces) foliated by marginally trapped compact spacelike submanifolds with co-dimension 2. The property of being trapped is related to the volume change of the submanifolds, and is encoded in the sign of the divergence or ``expansion'' of given null normal vector fields. Thus, the volume of closed (marginally) trapped submanifolds decreases (non-increases) initially along every possible direction of future evolution.

In the physical literature there is another characteristic, called {\em shear}, associated to horizons and evolution. For instance, the standard horizons of isolated black holes in equilibrium are Killing horizons \cite{Wald}, and they happen to be ``shear-free". The shear measures the local instantaneous deformation of a given submanifold when starting to evolve, while keeping its volume fixed. Mathematically, this is also associated to the extrinsic properties of the submanifold, and one can realize that the shear-free property corresponds to the submanifold being umbilical ---along the evolution direction. However, in contrast to the attention devoted to the mean curvature vector field and its properties ---e.g. \cite{AliasEstudilloRomero,BektasErgut,Chen1973,Chen1981,Chen2011} and references therein---, the part of the shape tensor that measures the referred  shear has not been considered previously, and the related umbilical properties scarcely considered ---except for the case of totally umbilical submanifolds.

Motivated by these facts, it is our purpose in this paper to introduce the extrinsic quantities associated to the mentioned deformation of submanifolds along normal directions, and to provide the relationship with their umbilical properties. We will do that in general semi-Riemannian manifolds for generic Riemannian submanifolds, keeping the dimension and co-dimension free.
To that end, we introduce the \emph{total shear tensor} as the trace-free part of the second fundamental form tensor and we call \emph{shear operators} the trace-free parts of the corresponding shape operators. We also introduce the {\em shear scalars} which allow us to make the link with the concepts in the physical literature. A new useful quadratic operator which is analogous to the Casorati operator \cite{Chen2011} but based on the shear quantities is also defined. 

Several notions of umbilical submanifolds are then considered. The classical umbilical property concerns co-dimension one surfaces, and refers to points that are ``spherical" in the sense that all {\em tangent} directions are indistinguishable there from the extrinsic point of view \cite{Eisenhart}. However, this is too demanding in higher co-dimension, because there are several normal directions and the submanifold can behave umbilically along some, but not along some other, directions. This leads to the notions of {\em totally umbilical} and, more importantly to our purposes, of {\em umbilical along some normal direction(s)}. The particular cases of pseudo-umbilical as well as the novel concepts of ortho-umbilical and sub-geodesic submanifolds are singled out.

While these developments are carried out in arbitrary dimension and co-dimension, in this paper we want to focus on the relevant case of co-dimension $2$ submanifolds. Our main goal is to characterize spacelike co-dimension $2$ submanifolds that are umbilical along a normal direction. More precisely, given an isometric immersion of a Riemannian $n$-manifold $\S$ into a semi-Riemannian $(n+2)$-manifold, we present necessary and sufficient conditions for $\S$ to be umbilical with respect to a normal direction.
The necessary and sufficient conditions we find are given in terms of the total shear tensor, or equivalently in terms of algebraic properties that have to be satisfied by any two shear operators.
We prove that the umbilical direction, if it exists, must be unique ---unless the submanifold is totally umbilical, of course. Moreover, by means of the total shear tensor we can provide a formula to compute it explicitly. When the ambient manifold is Lorentzian we also provide a way of determining its causal character. \\

Succinctly, the plan of the paper is as follows.
In section $2$ we introduce the notation and we recall some basic definitions. In section $3$ we give the definition of an umbilical submanifold, we specify several sub-cases and introduce our new shear objects. Section $4$ concentrates on general results for the special relevant case of co-dimension 2 that is to be assumed in the rest of the paper. In particular, we analyse the possibility of submanifolds which are both pseudo- and ortho-umbilical. In section $5$ we present some first results and characterize pseudo-umbilical submanifolds. Section $6$ is devoted to the main theorem. In section $7$ we determine the umbilical direction and characterize ortho-umbilical submanifolds. In section $8$ we concentrate on the Lorentzian case.  Finally, in section $9$ we give an example based on the geometry of the Kerr black hole for $n=2$.\\

Particular instances of the results in this paper and some of the underlying ideas were previously given, for the case of surfaces in 4-dimensional Lorentzian manifolds, in \cite{S}.

\section{Preliminaries: Basic concepts of submanifold theory}

Let $\S$ be an orientable $n$-dimensional manifold and $\Phi: \S \longrightarrow \M$
an immersion into an oriented $(n+k)$-dimensional semi-Riemannian manifold $(\M,\bar g)$. Assume that $g:= \Phi^{\star}\bar g $ is positive definite everywhere on $\S$, so that $(\S,g)$ is an oriented Riemannian manifold. Then $(\Phi(\S),\bar g)$ and $(\S,g)$ are isometric and we will always locally identify them. Then, $(\S, g)$ is called a {\em spacelike} submanifold of $(\M,\bar g)$.

If $\overline{\nabla}$ and $\nabla$ are the Levi-Civita connections of $(\M,\g)$ and $(\S,g)$ respectively, $X,Y \in \tangsp$ and $\xi \in \normsp$, then the formulas of Gauss and Weingarten give a decomposition of the vector fields $\overline{\nabla}_X Y$ and $\overline{\nabla}_X \xi$ in their tangent and normal components \cite{KobayashiNomizu,Kriele,ONeill}:
\begin{align*}
& \overline{\nabla}_X Y = \nabla_X Y + h(X,Y), \\
& \overline{\nabla}_X \xi = -A_{\xi}X + \nabla^{\perp}_X \xi.
\end{align*}
Here, $h(X,Y)=h(Y,X) \in \normsp$ for all $X,Y \in \tangsp$ where $h$ acts linearly (as a 2-covariant tensor) on its explicit arguments, and is called the \emph{second fundamental form} or \emph{shape tensor} of the immersion, $A_{\xi}$ is a self-adjoint operator called the \emph{shape operator} or \emph{Weingarten operator} associated to $\xi$ and $\nabla^{\perp}$ is a connection in the normal bundle. The relation between the former two is given by
\begin{align}\label{formula_Weingarten}
g(A_{\xi}X,Y) = \g(h(X,Y),\xi)
\end{align}
for all $X,Y \in \tangsp$ and all $\xi \in \normsp$. 

Given any orthonormal frame $\{e_1,\ldots,e_n\}$ in $\tangsp$, the \emph{mean curvature vector field} $H\in\normsp$ is defined as \cite{KobayashiNomizu,Kriele,ONeill}
\begin{align}\label{def0_H}
H = \frac{1}{n} \sum_{i=1}^n h(e_i,e_i).
\end{align}
The component of $H$ along a certain normal vector field $\xi \in \normsp$ up to a factor $n$ or, equivalently, the trace of the shape operator associated to $\xi$ is called the \emph{expansion of $\S$ along $\xi$}:
\begin{align}\label{definition_expansion}
\theta_{\xi} = n \, \g(H,\xi) = \tr (A_{\xi})
\end{align}
where $\tr$ denotes the trace.

\begin{remark}\label{remark_physicsnotation}
The terminology \emph{expansion} comes from the physics literature, see for example \cite{GlobalLorentzianGeometry,Kriele,S2007, S}. It should be noted that the factor $1/n$ in the definition of the mean curvature is often omitted in this literature.
\end{remark}

Given a local frame $\{\xi_1,\ldots,\xi_k\}$ in $\normsp$ which is orthonormal, i.e., $\bar g(\xi_i,\xi_j) = \epsilon_i \delta_{ij}$ with $\epsilon_{i}^2=1$ for all $i,j \in \{1,\ldots,k\}$, the \emph{Casorati operator}, see for example \cite{Chen2011}, is
defined by 
\begin{align}\label{def_B}
\B = \sum_{i=1}^k \g(\xi_i,\xi_i) A_{\xi_i}^2.
\end{align}
One can check that this definition does not depend on the chosen frame. Indeed, in the local orthonormal tangent frame $\{e_1,\ldots,e_n\}$, $\B$ is completely determined by
\begin{align}\label{formula_B}
g(\B X,Y) = \sum_{i=1}^n \g(h(X,e_i),h(Y,e_i))
\end{align}
for every $X,Y \in \tangsp$. Since all shape operators are self-adjoint, the same holds for the Casorati operator. 

Let $\T$ denote the set of all self-adjoint (1,1)-tensor fields on $\S$. We define the following pointwise positive-definite scalar product on $\T$:
\begin{align}\label{scproduct}
\scprod{A}{B} = \tr(AB) 
\end{align}
for all $A,B \in \T$.

\medskip


\section{Definitions}

\subsection{The total shear tensor, the shear operators and the shear scalars}

As far as we know, the following elementary  extrinsic objects have never been given a name in the literature.

\begin{definition}\label{def_tildeh}
Let $\Phi: (\S,g) \to (\M,\g)$ be the isometric immersion introduced above. Using the previous notations and conventions:
\begin{itemize}
\item The \emph{total shear tensor} $\tildeh$ is defined as the trace-free part of the second fundamental form:
\begin{align*}
\tildeh(X,Y)= h(X,Y) - g(X,Y) H.
\end{align*}
\item The \emph{shear operator} associated to $\xi \in \normsp$ is the trace-free part of the corresponding shape operator:
\begin{align*}
\tildeA_{\xi} = A_{\xi} - \frac 1n \theta_{\xi} \id,
\end{align*}
where $\id$ denotes the identity operator.
\item
The \emph{shear scalar} $\sigma_{\xi}$ associated to $\xi \in \normsp$ is defined up to sign as
\begin{align*}
\sigma_{\xi}^2 = \tr ( \tildeA_{\xi}{}^2 ).
\end{align*}
\end{itemize}
\end{definition}

It is clear that the total shear tensor and the shear operators are related by 
\begin{align*}
g(\tildeA_{\xi}X,Y) = \g(\tildeh(X,Y),\xi)
\end{align*}
for all $X,Y\in\tangsp$ and all $\xi\in\normsp$. The shear scalar was introduced in \cite{S}, yet in another way adapted to the case $n=2$. The alternative definition above is better suited for general dimension $n$ and works since $\tildeA_{\xi}$ is self-adjoint and hence the trace of $\tildeA_{\xi}^2$ is non-negative. We will come back to the ambiguity of the sign later in the paper. For now, notice that
\begin{align}\label{formula1_sigmaN}
\sigma_{\xi}^2 = \scprod{\tildeA_{\xi}}{\tildeA_{\xi}},
\end{align}
where $\scprod{\,}{\,}$ is the scalar product defined in (\ref{scproduct}), and that $\sigma_{\xi}=0$ if and only if $\tildeA_{\xi}=0$.

\begin{remark} 
The name \emph{total shear tensor} comes from the relation existing with the well known ``shear'' of the physics literature. In general relativity shear refers to one of the three kinematic quantities characterizing the flow of (usually timelike or null) vector fields, also called congruences, of a given Lorentzian manifold. The link arises because, if one such vector field is orthogonal to $\S$, then its shear on $\S$ would be given by $|\sigma_\xi|$. Notice that another of these quantities is the expansion that we introduced in (\ref{definition_expansion}). More about congruences and kinematic quantities can be found in \cite{GlobalLorentzianGeometry,HawkingEllis,Wald}.
\end{remark}

Given a local orthonormal frame $\{\xi_1,\ldots,\xi_k\}$ in $\normsp$, we can define the self-adjoint operator
\begin{align}\label{def_J}
\J = \sum_{i=1}^k \g(\xi_i,\xi_i) \tildeA_{\xi_i}^2.
\end{align}
Notice the analogy of definition (\ref{def_J}) with the one given in (\ref{def_B}) for the Casorati operator $\B$. Again, the above definition is frame independent and we have, in analogy with \eqref{formula_B},
\begin{align}\label{formula_J}
g(\J X,Y) = \sum_{i=1}^n \g(\tildeh(X,e_i),\tildeh(Y,e_i))
\end{align}
for any local orthonormal tangent frame $\{e_1,\ldots,e_n\}$ and any $X,Y \in \tangsp$.

\subsection{Several types of umbilicity}\label{section_definitionumbilical}

The concept of umbilical point is classical in Riemannian and semi-Riemannian geometry. Umbilical submanifolds have been extensively studied in the literature, but often only in certain ambient spaces, such as real space forms. For the very first works on the subject the reader can consult the references cited in \cite{S}. For a general overview, we refer to \cite{Chen1973,Chen1981} for the Riemannian setting and to \cite{Chen2011} for the semi-Riemannian one. In particular, results concerning pseudo-umbilical submanifolds (cfr. Definition \ref{definition_umbilicalsubmanifolds} below) in semi-Riemannian geometry can be found in \cite{AliasEstudilloRomero,BektasErgut,Cao,HuJiNiu,KimKim,SongPan,Sun}. 

When a submanifold has co-dimension one, that is when it is a hypersurface, a point can only be umbilical with respect to one normal direction. However, when the co-dimension of a submanifold is higher than one, there are several possible directions along which a point can be umbilical. We will define the different notions of umbilicity in this section with respect to a normal vector field, rather than with respect to a normal vector at a point, but all the definitions make sense when stated pointwise.

\begin{definition}\label{definition_umbilicalsubmanifolds}
Using the notations and conventions introduced above for the immersion $\Phi: (\S,g) \to (\M,\g)$, the submanifold $(\S,g)$ is said to be
\begin{itemize}
\item \emph{umbilical with respect to} $\xi\in\normsp$ if $A_{\xi}$ is proportional to the identity;
\item \emph{pseudo-umbilical} if it is umbilical with respect to the mean curvature $H$;
\item \emph{totally umbilical} if it is umbilical with respect to all $\xi\in\normsp$;
\item \emph{$\xi$-subgeodesic} if there exists $\xi\in\normsp$ such that $h(X,Y)=L(X,Y) \xi$ for all $X,Y\in \tangsp$, where $L$ is a symmetric $(0,2)$-tensor field  on $\S$.
\end{itemize}
\end{definition}

Observe that $\S$ is umbilical with respect to $\xi\in\normsp$ if and only if $A_{\xi} = (\theta_{\xi}/n) \id$ or, equivalently, $\tildeA_{\xi}=0$. If $\S$ is umbilical with respect to $\xi\in\normsp$, then it is umbilical with respect to all vector fields proportional to $\xi$. Therefore, we will often say that $\S$ is umbilical with respect to the normal \emph{direction} spanned by $\xi$ because this is a property that provides information about the umbilical direction span$(\xi)$ regardless of the length {\em and the orientation} of $\xi$. Finally, note that $\S$ is totally umbilical if and only if $h(X,Y)=g(X,Y)H$ for all $X,Y \in \tangsp$ or, equivalently, if and only if $\tildeh=0$.

The notion of $\xi$-subgeodesic submanifold was first introduced in \cite{S2007}. If $\S$ is $\xi$-subgeodesic for some $\xi\in\normsp$, the first normal space, i.e., the image of the second fundamental form, is at most one-dimensional at every point. It follows that all shape operators are proportional at points where $\xi$ does not vanish. Indeed, at such points one has
\begin{align} \label{allshapeopprop}
\g(\xi,\eta_2) A_{\eta_1} = \g(\xi,\eta_1) A_{\eta_2} 
\end{align}
for any $\eta_1,\eta_2 \in \normsp$. Furthermore, at points where $H\neq 0$, $\xi$-subgeodesic submanifolds have $\xi$ proportional to $H$, as can be seen by taking the trace of the equation $h(X,Y)=L(X,Y)\xi$. Therefore if $\S$ is $\xi$-subgeodesic, it is also $H$-subgeodesic.

Notice that, if $\S$ is $\xi$-subgeodesic, then any geodesic $\gamma:I\subseteq\mathbb{R} \rightarrow \S$ of $(\S,g)$ satisfies $\overline{\nabla}_{\gamma'}\gamma' = h(\gamma',\gamma') = f \xi$ for some function $f:I \to \mathbb{R}$, hence $\gamma$ is a subgeodesic with respect to $\xi$ in $(\M,\g)$, see \cite{Schouten}. This explains the terminology \emph{$\xi$-subgeodesic}.

\medskip


\section{The case of co-dimension $k=2$}

From now on we will assume that the immersion $\Phi:(\S,g)\rightarrow(\M,\bar g)$ has \emph{co-dimension two}.

\subsection{The structure of a normal bundle with two-dimensional fibers} \label{section_onthenormalbundle}

The normal bundle can have signature $(+,+)$, $(-,+)$ and $(-,-)$. We will write $(\epsilon_1,\epsilon_2)$, where $\epsilon_1^2=\epsilon_2^2=1$, in order not to specify one of them.
We denote by $\{\xi_1,\xi_2\}$ a local orthonormal frame in $\normsp$ with $g(\xi_i,\xi_i)=\epsilon_i$, $i\in\{1,2\}$. With respect to this frame, the second fundamental form $h$ decomposes as
\begin{align}\label{formula_h_ONbasis}
h(X,Y) = \epsilon_1 \, g(A_{\xi_1}X,Y) \, \xi_1 + \epsilon_2 \, g(A_{\xi_2}X,Y) \, \xi_2
\end{align}
for any $X,Y \in \tangsp$.
In terms of the frame $\{\xi_1,\xi_2\}$, the mean curvature vector can be expressed as
\begin{align}\label{def_H_ONbasis}
H = \frac 1n \left(\epsilon_1 \theta_{\xi_1} \xi_1 + \epsilon_2 \theta_{\xi_2} \xi_2 \right),
\end{align}
where we have used formula (\ref{formula_h_ONbasis}).

With the volume forms of the ambient manifold $\M$ and the submanifold $\S$, it is possible to define a volume form $\omega^\perp$ on the normal bundle. Using $\omega^\perp$, we can then define for any normal vector field $\xi\in\normsp$ its Hodge dual vector field $\hodge \xi\in\normsp$ by (see \cite{BHMS} for the Lorentzian case)
\begin{align} \label{def_hodge}
\g(\hodge \xi,\eta) = \omega^\perp (\xi,\eta)
\end{align}
for all $\eta\in\normsp$. The Hodge dual operator is a linear operator satisfying
\begin{align}\label{properties_hodge}
\hodge(\hodge \xi) = -\epsilon_1\epsilon_2 \, \xi,
\quad\quad
\g(\hodge \xi,\eta) = - \g(\xi,\hodge \eta)
\end{align}
for all $\xi,\eta\in\normsp$. In particular,
\begin{align*}
\g(\hodge \xi,\xi) = 0,
\quad\quad
\g(\hodge \xi,\hodge \xi) = \epsilon_1 \epsilon_2 \, \g(\xi,\xi)
\end{align*}
and if we assume that the orthonormal frame $\{\xi_1,\xi_2\}$ is oriented such that $\omega^{\perp}(\xi_1,\xi_2)=1$, then
\begin{align}\label{formula_hodgeu&v}
\hodge \xi_1 = \epsilon_2 \, \xi_2,
\quad\qquad
\hodge \xi_2 = -\epsilon_1 \, \xi_1.
\end{align}
Combining formulas (\ref{def_H_ONbasis}) and (\ref{formula_hodgeu&v}) the Hodge dual of the mean curvature vector field is
\begin{align}\label{def_hodgeH_ONbasis}
\hodge H = \frac{\epsilon_1\epsilon_2}{n} \left( \theta_{\xi_1} \xi_2 - \theta_{\xi_2} \xi_1 \right).
\end{align}
The vector field $\hodge H$ defines a (generically unique) direction with vanishing expansion:
\begin{align}\label{thetahodgeH}
\theta_{\hodge H} = \tr A_{\hodge H} = n \, \g(H,\hodge H) = 0.
\end{align}

The following notion, which we only define in co-dimension 2, was first introduced in \cite{S}.

\begin{definition}\label{definition_orthoumbilical}
Let $\Phi: (\S,g) \to (\M,\g)$ be an isometric immersion of a Riemannian manifold into a semi-Riemannian manifold with co-dimension $2$ and use the notations and conventions introduced above. The submanifold is said to be \emph{ortho-umbilical} if $A_{\star^\bot H}=0$.
\end{definition}

The terminology \emph{ortho-umbilical} is explained by the fact that the condition $A_{\star^\bot H}=0$ is actually equivalent to $\tilde A_{\hodge H}=0$, that is, to requiring that the submanifold is umbilical with respect to the vector field $\star^\bot H$ orthogonal to $H$, since we know that $\theta_{\hodge H}=0$, cfr. (\ref{thetahodgeH}).

The question arises of whether a submanifold can be pseudo- and ortho-umbilical at the same time. This is answered in the following Lemma.

\begin{lemma}\label{lemma:pseudo+ortho}
Let $\Phi: (\S,g) \to (\M,\g)$ be an isometric immersion of a Riemannian manifold into a semi-Riemannian manifold with co-dimension $2$. If $\S$ is both pseudo-umbilical and ortho-umbilical then at any point either
\begin{enumerate}
\item $(\S,g)$ is totally umbilical, or
\item the mean curvature vector field satisfies $\g(H,H)=0$.
\end{enumerate} 
\end{lemma}
\begin{proof}
If $(\S,g)$ is totally umbilical the result is trivial. Similarly, if $H=0$ at a point the result is empty. Otherwise, consider $H\neq 0$ and $\tilde h\neq 0$. From $A_{\star^\bot H}=0$ we deduce that $h(X,Y)$, and a fortiori $\tilde h(X,Y)$, points along $H$ for all $X,Y\in \tangsp$, and $\tilde A_{H}=0$ gives then
$$
0=g(\tilde A_H X,Y) =\bar g(\tilde h(X,Y),H) =\tilde L(X,Y) \bar g(H,H)
$$
with $\tilde L(X,Y) \neq 0$, hence $\bar g(H,H)=0$. 
\end{proof}

\subsection{Equivalence between ortho-umbilical and $\xi$-subgeodesic submanifolds}

In the following Proposition 
we prove that the property of being $\xi$-subgeodesic, introduced in Definition \ref{definition_umbilicalsubmanifolds}, is equivalent to the property of being ortho-umbilical when the co-dimension is two.

\begin{proposition}\label{proposition_Nsubgeodesic}
Let $\Phi: (\S,g) \to (\M,\g)$ be an isometric immersion of a Riemannian manifold into a semi-Riemannian manifold with co-dimension $2$. Then the following two conditions are equivalent on any open set where $H\neq 0$:
\begin{itemize}
\item[(1)] $\S$ is ortho-umbilical;
\item[(2)] $\S$ is $\xi$-subgeodesic for some non-zero $\xi\in\normsp$.
\end{itemize}
\end{proposition}

\begin{proof}
To prove the implication (1)$\Longrightarrow$(2), it suffices to take $\xi=H$. Conversely, assume that $\S$ is $\xi$-subgeodesic. Then $h(X,Y)$ is everywhere proportional to $\xi$ for any choice of $X,Y \in \tangsp$ ---so that in particular $A_{\hodge \xi} = 0$. As previously explained $H$ and $\xi$ are proportional, and thus $A_{\hodge \xi} = 0$ implies $A_{\hodge H} = 0$.
\end{proof}

\begin{remark}\label{remarkH=0}
If $H$ vanishes, at most, on a subset with empty interior, then Proposition \ref{proposition_Nsubgeodesic} is true globally on $\S$. Indeed, in Proposition \ref{proposition_Nsubgeodesic} we have proven the equivalence on any subset on which $H$ vanishes nowhere. Since the union of all such subsets is dense in $\S$, the result follows by a continuity argument.
\end{remark}

\noindent
By Proposition \ref{proposition_Nsubgeodesic} and formula (\ref{allshapeopprop}) it follows that the submanifold $\S$ is ortho-umbilical if and only if all shape operators are proportional to each other.

The following is a direct consequence of Proposition \ref{proposition_Nsubgeodesic} and its proof.

\begin{corollary}\label{corollary_orthoWeingartenvanishing}
On any open set where $H\neq 0$ there exists a non-zero normal vector field $\xi\in\normsp$ such that $A_{\hodge \xi}=0$ if and only if $\S$ is ortho-umbilical.
\end{corollary}

\medskip


\section{First results}

\subsection{Characterization of being pseudo-umbilical}

The Casorati operator $\B$, defined in (\ref{def_B}), and the operator $\J$, defined in (\ref{def_J}), are related to each other via the relation described in the following Lemma.

\begin{lemma}\label{lemma_relationBJ}
Let $\Phi: (\S,g) \to (\M,\g)$ be an isometric immersion of an $n$-dimensional Riemannian manifold into a semi-Riemannian manifold with co-dimension $2$. Let $\B$ be the Casorati operator and let $\J$ be the operator defined in (\ref{def_J}).
Then
\begin{align} \label{BminJ}
\B - \J =  2 \, \tildeA_H + \g(H,H) \id,
\end{align}
where $H$ is the mean curvature vector field. Moreover, $\emph{\tr}\left(\B-\J\right) = n \, \g(H,H)$.
\end{lemma}

\begin{proof}
The expression for the trace of $\B-\J$ follows immediately from \eqref{BminJ}, since $\tildeA_H$ is a trace-free operator. To prove \eqref{BminJ}, first observe that for any $\xi \in \normsp$
\begin{align} \label{for1}
A_{\xi}^2 - \tildeA_{\xi}^2 = A_{\xi}^2 - (A_{\xi} - \frac 1n \theta_{\xi} \id)^2 = \frac 2n \theta_{\xi} A_{\xi} - \frac{1}{n^2} \theta_{\xi}^2 \id.
\end{align}
Moreover, from (\ref{def_H_ONbasis}), we obtain
\begin{align} \label{for2}
A_H = \frac 1n (\epsilon_1 \theta_{\xi_1} A_{\xi_1} + \epsilon_2 \theta_{\xi_2} A_{\xi_2})
\end{align}
and 
\begin{align} \label{for3}
\g(H,H) = \frac{1}{n^2} (\epsilon_1 \theta_{\xi_1}^2 + \epsilon_2 \theta_{\xi_2}^2).
\end{align}
By using the definitions of $\B$ and $\J$ and formulas \eqref{for1}--\eqref{for3}, we then obtain
\begin{align*}
\B-\J & = \epsilon_1 (A_{\xi_1}^2 - \tildeA_{\xi_1}^2) + \epsilon_2 (A_{\xi_2}^2 - \tildeA_{\xi_2}^2) \\
& = \frac 2n (\epsilon_1 \theta_{\xi_1} A_{\xi_1} + \epsilon_2 \theta_{\xi_2} A_{\xi_2}) - \frac{1}{n^2} (\epsilon_1 \theta_{\xi_1}^2 + \epsilon_2 \theta_{\xi_2}^2) \id \\
& = 2 A_H - \g(H,H) \id.
\end{align*}
It now suffices to use the definition of the shear operator and the fact that $\theta_H = n \, \g(H,H)$ to conclude the proof.
\end{proof}

\begin{corollary}\label{corollary_pseudo}
Let $\Phi: (\S,g) \to (\M,\g)$ be an isometric immersion of an $n$-dimensional Riemannian manifold into a semi-Riemannian manifold with co-dimension $2$. Then $\S$ is pseudo-umbilical if and only if 
$$\B-\J=A_H.$$
Or equivalently, if and only if $\B-\J$ is proportional to the identity.
\end{corollary}

\begin{proof}
In the proof of Lemma \ref{lemma_relationBJ}, we obtained $\B-\J = 2 A_H - \g(H,H) \id$. Hence the formula $\B-\J=A_H$ is equivalent to $A_H = \g(H,H) \id$, which expresses exactly that the submanifold is pseudo-umbilical.
\end{proof}

When $n=2$, that is when the ambient manifold $\M$ has dimension $4$ and the submanifold $\S$ is a surface, the necessary and sufficient condition for $\S$ to be pseudo-umbilical is the Casorati operator $\B$ being proportional to the identity. This was proven in \cite{S}. (More precisely, it was proven in the Lorentzian case, but as the author explains in the final comments, the same proofs hold in other signature settings too.) In higher dimension, the situation is different:
although the property of both $\B$ and $\J$ being proportional to the identity is sufficient to prove that $\S$ is pseudo-umbilical (this follows from Corollary \ref{corollary_pseudo}), it is not necessary.

\medskip


\section{Main theorem}

\begin{theorem}\label{maintheorem}
Let $\Phi: (\S,g) \to (\M,\g)$ be an isometric immersion of an $n$-dimensional Riemannian manifold into a semi-Riemannian manifold with co-dimension $2$. Then the following conditions are all equivalent.
\begin{itemize}
\item[(i)] $\S$ is umbilical with respect to a non-zero normal vector field $\xi\in\normsp$.
\item[(ii)] Any two shear operators are proportional to each other.
\item[(iii)] There exist $\tildeA \in \T$ and $G \in \normsp$ such that $\langle \tildeA,\tilde A \rangle = n^2$ and 
\begin{align} \label{MTfor1}
\tildeh(X,Y) = g(\tildeA X,Y) \, G
\end{align} 
for all $X,Y\in\tangsp$.
\item[(iv)]
The components of any two shear operators $\tildeA_{\eta_1}$ and $\tildeA_{\eta_2}$ with respect to any tangent frame satisfy
\begin{align} \label{MTfor2}
(\tildeA_{\eta_1})^i_j \, (\tildeA_{\eta_2})^r_s = (\tildeA_{\eta_2})^i_j \, (\tildeA_{\eta_1})^r_s
\end{align}
for all $i,j,r,s=1,\ldots,n$.
\item[(v)]
Any two shear operators $\tildeA_{\eta_1}$ and $\tildeA_{\eta_2}$ satisfy
\begin{align} \label{MTfor3}
\scprod{\tildeA_{\eta_1}}{\tildeA_{\eta_2}}^2 = \sigma_{\eta_1}^2 \, \sigma_{\eta_2}^2.
\end{align}
\end{itemize}
\end{theorem}

\begin{proof}
(i) $\Longrightarrow$ (ii). Let $\eta \in \normsp$ be linearly independent from $\xi$. For any $\eta_1,\eta_2 \in \normsp$, there exist functions $a_1$, $b_1$, $a_2$ and $b_2$ such that $\eta_1 = a_1 \xi + b_1 \eta$ and $\eta_2 = a_2 \xi + b_2 \eta$. Since the correspondence $\bullet \mapsto \tildeA_{\bullet}$ is linear and $\tildeA_{\xi}=0$, we find $\tildeA_{\eta_1} = b_1 \tildeA_{\eta}$ and $\tildeA_{\eta_2} = b_2 \tildeA_{\eta}$. In particular, $\tildeA_{\eta_1}$ and $\tildeA_{\eta_2}$ are proportional.
\smallskip

(ii) $\Longrightarrow$ (iii). Let $\{\xi_1,\xi_2\} \subset \normsp$ be a orthonormal frame such that $\g(\xi_i,\xi_j)=\epsilon_i\delta_{ij}$. Since $\tildeA_{\xi_1}$ and $\tildeA_{\xi_2}$ are proportional, there exist an $\tildeA \in \T$ and functions $\lambda_1$ and $\lambda_2$ such that $\tildeA_{\xi_1} = \lambda_1 \tildeA$ and $\tildeA_{\xi_2} = \lambda_2 \tildeA$.  Then, by orthonormal expansion, we have
\begin{align*}
\tildeh(X,Y) &= \epsilon_1 \g(\tildeh(X,Y),\xi_1) \xi_1 + \epsilon_2 \g(\tildeh(X,Y),\xi_2) \xi_2 \\
&= \epsilon_1 g(\tildeA_{\xi_1}X,Y) \xi_1 + \epsilon_2 g(\tildeA_{\xi_2}X,Y) \xi_2 \\
&= g(\tildeA X,Y) (\epsilon_1 \lambda_1 \xi_1 + \epsilon_2 \lambda_2 \xi_2)
\end{align*}
for any $X,Y \in \tangsp$. It now suffices to choose $G$ in the direction of $\epsilon_1 \lambda_1 \xi_1 + \epsilon_2 \lambda_2 \xi_2$. By rescaling $G$ if necessary, we can ensure that $\langle \tildeA,\tildeA \rangle = \tr(\tildeA^2)=n^2$.
\smallskip

(iii) $\Longrightarrow$ (i). If $G=0$, then $\tildeh=0$ and the submanifold is totally umbilical. If $G \neq 0$, then also $\hodge G \neq 0$ and $\langle \tildeh(X,Y),\hodge G \rangle = 0$ for any $X,Y \in \tangsp$ implies that $\tildeA_{\hodge G}=0$ or, equivalently, that $\S$ is umbilical with respect to $\hodge G$.
\smallskip

(ii) $\Longleftrightarrow$ (iv). First, assume that (ii) is satisfied and choose any $\eta_1,\eta_2 \in \normsp$. If $\tildeA_{\eta_2}=0$, condition (iv) is satisfied. If $\tildeA_{\eta_2} \neq 0$, there exists a function $\lambda$ such that $\tildeA_{\eta_1} = \lambda \tildeA_{\eta_2}$ and both sides of \eqref{MTfor2} equal $\lambda (\tildeA_{\eta_2})^i_j (\tildeA_{\eta_2})^r_s$. Conversely, assume that (iv) is satisfied and choose $\eta_1,\eta_2 \in \normsp$. If $\tildeA_{\eta_2}=0$, condition (ii) is satisfied. If $\tildeA_{\eta_2} \neq 0$, there is at least one component, say $(\tildeA_{\eta_2})^i_j$, which is non-zero. Since \eqref{MTfor2} holds for any $r$ and $s$, it implies $\tildeA_{\eta_1} = ((\tildeA_{\eta_1})^i_j/(\tildeA_{\eta_2})^i_j) \tildeA_{\eta_2}$.
\smallskip

(ii) $\Longleftrightarrow$ (v). Since $\sigma_{\xi}^2 = \langle \tildeA_{\xi},\tildeA_{\xi} \rangle$ for any $\xi \in \normsp$, this equivalence is a direct consequence of the inequality of Cauchy-Schwarz.
\end{proof}

By Theorem \ref{maintheorem}, whenever there is an umbilical direction there is a normal vector field $G$ satisfying condition (iii). Notice however that $G$ is only determined up to sign. But this is natural because, as we already mentioned, being umbilical is a property related to a direction, not to a particular vector field. Using condition (iii), we have $g(\tildeA_{\xi}X,Y) = \g(\tildeh(X,Y),\xi) = g(\tildeA X,Y) \g(G,\xi)$ for all $X,Y\in\tangsp$ and all $\xi \in \normsp$. Hence 
\begin{align} \label{for6.1}
\tildeA_{\xi} = \g(G,\xi) \tildeA
\end{align}
and the corresponding shear scalar is given by $\sigma_{\xi}^2
= \tr ( \tildeA_{\xi}^2 )
= \g(G,\xi)^2 \, \tr ( \tildeA^2 )
= n^2 \, \g(G,\xi)^2$.
Since both $\sigma_{\xi}$ and $G$ are defined up to sign, we can set
\begin{align}\label{formula2_sigmaN}
\sigma_{\xi} = n \, \g(G,\xi)
\end{align}
for all $\xi\in\normsp$. Combining \eqref{for6.1} and \eqref{formula2_sigmaN} yields
\begin{align*}
\tildeA_{\xi} = \frac{\sigma_{\xi}}{n} \tildeA
\end{align*}
for all $\xi\in\normsp$, from which we can deduce
\begin{align*}
\scprod{\tildeA_{\eta_1}}{\tildeA_{\eta_2}} = \sigma_{\eta_1} \, \sigma_{\eta_2}
\end{align*}
for all $\eta_1,\eta_2\in\normsp$.

\begin{remark} \label{remarkMT}
Item (iii) of Theorem \ref{maintheorem} can be restated as follows: 
\begin{align*}
\tildeh(X,Y)^\flat \wedge \tildeh(Z,W)^\flat = 0
\end{align*}
for every $X,Y,Z,W\in\tangsp$, where $\wedge$ is the wedge product of one-forms and $\flat$ denotes the musical isomorphism: if $X$ is a vector field on $(\M,\g)$, then its associated one-form $X^\flat$ is given by $X^\flat(Y)=\g(X,Y)$ for every vector field $Y$ on $\M$.
\end{remark}

\begin{corollary}\label{corollary_weingartencommute}
Let $\Phi: (\S,g) \to (\M,\g)$ be an isometric immersion of an $n$-dimensional Riemannian manifold into a semi-Riemannian manifold with co-dimension $2$.
If $\S$ is umbilical with respect to a non-zero normal vector field then any two shape operators commute.
\end{corollary}

\begin{proof}
By condition (ii) of Theorem \ref{maintheorem} it follows that any two shear operators commute. It is easily seen that $[\tildeA_{\eta_1},\tildeA_{\eta_2}] = 0$ if and only if $[A_{\eta_1},A_{\eta_2}] = 0$ for any $\eta_1,\eta_2\in\normsp$.
\end{proof}

A consequence of Corollary \ref{corollary_weingartencommute} is that at any point of the submanifold there exists a (generically unique) orthonormal basis of the tangent space for which all shape operators diagonalize simultaneously. 

The converse of Corollary \ref{corollary_weingartencommute} is in general not true. However, it is true when the dimension of the ambient manifold $\M$ is $4$ and $\S$ is a surface ($n=2$), as described in the next corollary.

\begin{corollary}\label{corollary_n=2}
A necessary and sufficient condition for a spacelike surface in a $4$-dimensional semi-Riemannian manifold to be umbilical with respect to a non-zero normal direction is that any two shape operators commute.
\end{corollary}

\begin{proof}
The necessity of the condition follows from Corollary \ref{corollary_weingartencommute}. To prove that the condition is also sufficient, choose any two normal vector fields $\xi$ and $\eta$. Then $A_\xi$ and $A_\eta$ commute, such that both operators can be diagonalized simultaneously. Denote by $\lambda_1,\lambda_2$ and $\mu_1,\mu_2$ the eigenvalues of $A_\xi$ and $A_\eta$ respectively.
In an orthonormal frame the corresponding shear operators are then given by
\begin{align*}
\tildeA_\xi
= \frac{1}{2} \left(
\begin{array}{cc}
\lambda_1 - \lambda_2 & 0 \\
0 & \lambda_2 - \lambda_1
\end{array}
\right),
\qquad
\tildeA_\eta
= \frac{1}{2} \left(
\begin{array}{cc}
\mu_1 - \mu_2 & 0 \\
0 & \mu_2 - \mu_1
\end{array}
\right).
\end{align*}
It is now easily seen that $(\tildeA_\xi)^i_j (\tildeA_\eta)^r_s = (\tildeA_\eta)^i_j (\tildeA_\xi)^r_s$ for any $i,j,r,s \in \{1,2\}$ such that, by Theorem \ref{maintheorem}, the surface is umbilical with respect to some non-zero normal vector field.
\end{proof}

Corollary \ref{corollary_n=2} was already proven in \cite{S} in the case of a Lorentzian ambient space.

\medskip


\section{Consequences of the main theorem}

\subsection{The umbilical direction}

If $\{\xi_1,\xi_2\}$ is an orthonormal frame in the normal bundle with $\g(\xi_i,\xi_j) = \epsilon_i \delta_{ij}$, one can deduce from (\ref{formula2_sigmaN}) the following explicit expression for $G$:
\begin{align}\label{formula_G_ONbasis}
G = \frac 1n (\epsilon_1 \sigma_{\xi_1} \xi_1 + \epsilon_2 \sigma_{\xi_2} \xi_2).
\end{align}

\begin{corollary}\label{corollary_uniqueness}
Let $\Phi: (\S,g) \to (\M,\g)$ be an isometric immersion of an $n$-dimensional Riemannian manifold into a semi-Riemannian manifold with co-dimension $2$. If $\S$ is umbilical with respect to a normal direction, then such a direction is unique and it is spanned by $\hodge G$, unless $G=0$, in which case $\S$ is totally umbilical.
\end{corollary}

\begin{proof}
Suppose that $\S$ is umbilical with respect to a non-zero vector field $\xi\in\normsp$. This means that its shear operator vanishes, $\tildeA_{\xi} = 0$, or, equivalently, $\g(\tildeh(X,Y),\xi) = 0$ for every $X,Y\in\tangsp$. From Theorem \ref{maintheorem} (iii) it then follows that $G$ and $\xi$ are orthogonal and thus $\xi$ has to be proportional to $\hodge G$.
If $G=0$, then $\tildeh=0$ and $\S$ is totally umbilical.
\end{proof}

From Corollary \ref{corollary_uniqueness}, \eqref{formula_hodgeu&v} and (\ref{formula_G_ONbasis}) one obtains an explicit expression for the umbilical direction:
\begin{align}\label{formula_hodgeG}
\hodge G = \frac{\epsilon_1 \epsilon_2}{n} (\sigma_{\xi_1} \, \xi_2 - \sigma_{\xi_2} \, \xi_1).
\end{align}
It is possible to find other expressions for the umbilical direction in terms of the eigenvalues of the shape operators $A_{\xi_1}$ and $A_{\xi_2}$. We know that if an umbilical direction exists, these two operators can be diagonalized simultaneously. Let $\lambda_i$ and $\mu_i$ ($i=1,\ldots,n$) denote the eigenvalues of $A_{\xi_1}$ and $A_{\xi_2}$ respectively. Then
$\lambda_i - \theta_{\xi_1}/n$ and $\mu_i - \theta_{\xi_2}/n$ ($i=1,\ldots,n$) are the eigenvalues of the shear operators $\tildeA_{\xi_1}$ and $\tildeA_{\xi_2}$. We know that there exist functions $a_1$ and $a_2$ such that $\tildeA_{a_1\xi_1+a_2\xi_2} = a_1 \tildeA_{\xi_1} + a_2 \tildeA_{\xi_2} = 0$. Obviously, $(a_1,a_2)$ has to be proportional to $(\mu_i - \theta_{\xi_2}/n, -(\lambda_i - \theta_{\xi_1}/n) )$ for any $i=1,\ldots,n$. Hence
\begin{align}\label{formula_Numb}
\eta_i
= \left( \mu_i - \frac{\theta_{\xi_2}}{n} \right) \xi_1 - \left( \lambda_i - \frac{\theta_{\xi_1}}{n} \right) \xi_2
\end{align}
is a normal vector field with respect to which $\S$ is umbilical for any $i=1,\ldots,n$. All these vector fields are proportional to each other and to $\hodge G$. Moreover, using \eqref{formula_hodgeG} and \eqref{formula_Numb}, one sees
\begin{align*}
\sum_{i=1}^n \g(\eta_i,\eta_i) 
&= \epsilon_1 \sum_{i=1}^n \left(\mu_i - \frac{\theta_{\xi_2}}{n}\right)^2 + \epsilon_2 \sum_{i=1}^n \left(\lambda_i - \frac{\theta_{\xi_1}}{n}\right)^2 \\
&= \epsilon_1 \tr(\tildeA_{\xi_2}^2) + \epsilon_2 \tr(\tildeA_{\xi_1}^2) \\
&= \epsilon_1 \sigma_{\xi_2}^2 + \epsilon_2 \sigma_{\xi_1}^2 \\
&= n^2 \, \g(\hodge G,\hodge G).
\end{align*}

\subsection{Characterization of $\S$ being ortho-umbilical} 

For the notation used in the following corollary we refer to Remark \ref{remarkMT}.

\begin{corollary}\label{corollary_ortho}
Let $\Phi: (\S,g) \to (\M,\g)$ be an isometric immersion of an $n$-dimensional Riemannian manifold into a semi-Riemannian manifold with co-dimension $2$. On any open set where $H$ does not vanish, $\S$ is ortho-umbilical if and only if 
\begin{align*}
h(X,Y)^\flat \wedge H^\flat = 0
\end{align*}
for every $X,Y\in\tangsp$.
\end{corollary}

\begin{proof}
Suppose that $\S$ is ortho-umbilical, i.e., that $A_{\hodge H}=0$. Then, from Propostion \ref{proposition_Nsubgeodesic} we know that $\S$ is $H$-subgeodesic so that $h(X,Y)$ is indeed proportional to the mean curvature vector field $H$ for every $X,Y\in\tangsp$.

Conversely, suppose that $h(X,Y)^\flat \wedge H^\flat = 0$ for every $X,Y\in\tangsp$ and that $H\neq 0$ at a point in $\S$, then $h(X,Y)=L(X,Y)H$ for all $X$ and $Y$ tangent to $\S$ at that point, and this implies $A_{\hodge H}=0$ there.
\end{proof}
A comment similar to Remark \ref{remarkH=0} applies here too, as it is obvious.

It is worth noticing that the property of $\S$ being ortho-umbilical is somehow special. In fact, it implies that $\S$ is $H$-subgeodesic and also that the total shear tensor $\tildeh$ and the second fundamental form tensor $h$ are always proportional to each other and to $H$.

\medskip


\section{The Lorentzian case}\label{section_Lorentziancase}

In this section we will always assume the normal bundle has signature $(-,+)$ or, equivalently, that $(\M,\g)$ is a Lorentzian manifold with signature $(-,+,\ldots,+)$.

\subsection{Null frame in $\normsp$}

In this setting one can choose a frame $\{k,\ell\}$ consisting of null vector fields in the normal bundle $\normsp$, so that
\begin{align} \label{g_kl}
\g(k,k) = \g(\ell,\ell) = 0, \qquad \g(k,\ell) = -1,
\end{align}
the last equality being a convenient normalization condition. For any $X,Y \in\tangsp$, the second fundamental form $h$ and the mean curvature vector field $H$ decompose in the null frame as
\begin{align}
h(X,Y)
&= - \, g(A_k X,Y) \ell - g(A_\ell X,Y) k, \label{formula_tildeh_nullbasis}\\
H
&= \frac{1}{n} \left( - \, \theta_k \, \ell - \theta_\ell \, k \right). \label{def_H_nullbasis}
\end{align}
The quantities $\theta_k = \tr (A_k)$ and $\theta_\ell = \tr (A_\ell)$, are called \emph{null expansions}.

Using formula \eqref{formula_B} one can prove that the Casorati operator $\B$ is minus the anti-commutator of the two null shape operators:
\begin{align} \label{Basanticomm}
\B = - \{ A_k , A_{\ell} \}
= - A_k A_\ell - A_\ell A_k.
\end{align}

After changing the order if necessary, we may assume that the frame $\{k,\ell\}$ is positively oriented, i.e., that $\omega^{\perp}(k,\ell)=1$. It then follows from \eqref{def_hodge} that
\begin{align} \label{hodge_kl}
\hodge k = - k, \quad \quad \hodge \ell = \ell
\end{align}
and from \eqref{def_H_nullbasis} that
\begin{align}\label{def_hodgeH_nullbasis}
\hodge H = \frac{1}{n} \left( - \, \theta_k \, \ell + \theta_\ell \, k \right).
\end{align}
We also have $\g(\hodge\xi,\hodge\xi) = - \g(\xi,\xi)$ for every $\xi \in \normsp$, and in particular
\begin{align}\label{formula:g(H,H)=-g(hodgeH,hodgeH)}
\g(\hodge H,\hodge H) = -\g(H,H)= \frac{2}{n^2}\theta_k \theta_{\ell}.
\end{align}
The total shear tensor decomposes as
\begin{align}\label{formula_tildeh}
\tildeh(X,Y) = - \g(\tildeA_k X,Y) \ell - \g(\tildeA_\ell X,Y) k
\end{align}
for every $X,Y \in \tangsp$, where the operators $\tildeA_k$ and $\tildeA_{\ell}$ are called \emph{null} shear operators. Moreover, the functions $\sigma_k$ and $\sigma_{\ell}$ are the \emph{null} shear scalars. Using formula (\ref{formula_tildeh}), one can prove that the operator $\J$ equals minus the anti-commutator of the two null shear operators:
\begin{align}\label{formula_J_nullbasis}
\J = - \, \{ \tildeA_k , \tildeA_{\ell} \}.
\end{align}
\noindent
Notice the analogy between this formula and \eqref{Basanticomm}.

\subsection{The causal character of the umbilical direction}

Assuming there exists an umbilical direction, it follows from \eqref{formula2_sigmaN}, \eqref{g_kl} and \eqref{hodge_kl} that the vector fields $G$ and $\hodge G$ can be expressed as
\begin{align}\label{formulas_GhodgeG_nullbasis}
G = - \frac 1n (\sigma_{\ell} k + \sigma_k \ell),
\qquad \quad
\hodge G = \frac 1n (\sigma_{\ell} k - \sigma_k \ell).
\end{align}
A way to determine the sign of $\g(\hodge G, \hodge G)$ is by considering the operators $\J$ and $\B$. By using that $\tildeh(X,Y)=g(\tildeA X,Y)G$ and hence $h(X,Y) = g(\tildeA X,Y)G + g(X,Y)H$ for every $X,Y \in \tangsp$ in combination with formulas \eqref{formula_J} we obtain
\begin{align*}
\J = \g(G,G) \tildeA^2.
\end{align*}
Taking the trace gives $\tr(\J) = n^2 \g(G,G)$ and hence
\begin{align*}
\g(\hodge G, \hodge G) = -\g(G,G) = -\frac{1}{n^2} \tr(\J).
\end{align*}
By formula (\ref{formula_J_nullbasis}) we obtain
\begin{align*}
\g(\hodge G,\hodge G)
= \frac{2}{n^2} \tr(\tildeA_k \tildeA_\ell)
=\frac{2}{n^2} \scprod{\tildeA_k}{\tildeA_\ell}
\end{align*}
and thus we have
\begin{align*}
& \scprod{\tildeA_k}{\tildeA_\ell} < 0 \ \Rightarrow \ \hodge G \text{ is timelike}, \\
& \scprod{\tildeA_k}{\tildeA_\ell} > 0 \ \Rightarrow \ \hodge G \text{ is spacelike},\\
& \scprod{\tildeA_k}{\tildeA_\ell} = 0 \ \Rightarrow \ \hodge G \text{ is null.}
\end{align*}
Using (\ref{BminJ}) we also get $\tr(\B) = n^2\g(G,G) + n\g(H,H)$ so that
\begin{align*}
\g(\hodge G, \hodge G) = -\frac{1}{n^2} (\tr(\B)-n\g(H,H))
\end{align*}
which reproves the same result found in \cite{S}.
All this implies
\begin{align*}
& \tr(\J) < 0 \ \Rightarrow \ \hodge G \text{ is spacelike}, && \tr(\B) < n\g(H,H) \ \Rightarrow \ \hodge G \text{ is spacelike},\\
& \tr(\J) > 0 \ \Rightarrow \ \hodge G \text{ is timelike},  && \tr(\B) > n\g(H,H) \ \Rightarrow \ \hodge G \text{ is timelike},\\
& \tr(\J) = 0 \ \Rightarrow \ \hodge G \text{ is null},      && \tr(\B) = n\g(H,H) \ \Rightarrow \ \hodge G \text{ is null}.
\end{align*}

\subsection{Submanifolds which are both pseudo- and ortho-umbilical}

From Lemma \ref{lemma:pseudo+ortho} we know that the only interesting case when $\S$ can be pseudo- and ortho-umbilical at the same time arises in the Lorentzian signature we are considering now. Thus, we analyze this case in a little more detail.

When $\S$ is not totally umbilical, we have
\begin{proposition}
Let $\Phi: (\S,g) \to (\M,\g)$ be an isometric immersion of a Riemannian manifold into a Lorentzian manifold with co-dimension $2$. The three following conditions are equivalent at any point $p\in \S$ where $H\neq 0$ and $\S$ is not totally umbilical:
\begin{enumerate}
\item $\B -\J=0$,
\item $(\S,g)$ is both pseudo-umbilical and ortho-umbilical,
\item $\B=0$ and $\J=0$.
\end{enumerate}
Furtheremore, in all cases we have $\g(H,H)=0$ at $p$.
\end{proposition}

\begin{proof}
$(1) \Rightarrow (2)$: Assume $\B=\J$. Then, by Lemma \ref{lemma_relationBJ}, we obtain $2 \tildeA_H + \g(H,H) \id=0$. Taking the trace of this formula gives $\g(H,H)=0$ and hence also $\tildeA_H=0$. Therefore, $\S$ is pseudo-umbilical at $p$ and $H$ is a non-zero null vector there, so that from \eqref{formula:g(H,H)=-g(hodgeH,hodgeH)}  follows that $\hodge H$ is also null and, being orthogonal to $H$, proportional to $H$. Thus $\tildeA_{\hodge H}=0$ too.

\noindent $(2) \Rightarrow (3)$: Since $\S$ is ortho-umbilical at $p$ and as $H\neq 0$ there, Corollary \ref{corollary_ortho} implies the existence of a $2$-covariant symmetric tensor $L$ such that 
\begin{equation} \label{h_ortho}
h(X,Y)=L(X,Y)H 
\end{equation}
for every $X,Y\in T_p\S$. On the other hand, since $\S$ is pseudo-umbilical, we have $\g(\tildeh(X,Y),H) = 0$. Using the definition of $\tildeh$ and \eqref{h_ortho}, this condition reduces to
\begin{align}\label{corollary_pseudo+ortho_formula3}
\left( L(X,Y) - g(X,Y) \right) \g(H,H) = 0
\end{align}
for every $X,Y\in T_p\S$.
If $\g(H,H)$ did not vanish, then we would have $L=g$ at $p$ and, by \eqref{h_ortho}, $\S$ would be totally umbilical there against hypothesis. Thus, we deduce $\g(H,H) = 0$ at $p$. Using this together with $\tildeA_{H}=0$ in Lemma \ref{lemma_relationBJ} we derive $\B-\J =0$ at $p$. Now, we can compute the Casorati operator to check that it actually vanishes at $p$ (and therefore so does $\J$): if $\{e_1,\ldots,e_n\}$ is a local orthonormal basis in $T_p\S$, by formula (\ref{formula_B}) we have
\begin{align*}
g(\B(X),Y)
= \sum_{i=1}^n g( L(X,e_i)H , L(Y,e_i)H )
= \g(H,H) \sum_{i=1}^n L(X,e_i) L(Y,e_i) = 0.
\end{align*}

\noindent $(3) \Rightarrow (1)$: Trivial.
\end{proof}

From a physical point of view, the Lorentzian $n=2$ case is of particular interest. In fact, spacelike surfaces immersed in a \emph{spacetime} (a $4$-dimensional Lorentzian manifold) satisfying the property of having $H$ null everywhere on the surface are of extreme importance in the framework of gravitational theories. A complete classification of these surfaces can be found in \cite{S2007}. In the following section we will show an example of those.

\medskip


\section{Example: Kerr spacetime}
The Kerr spacetime can be locally described as the manifold $\mathbb{R}^2\times S^2$ endowed with the Lorentzian metric \cite{Kerr}
\begin{align*}
\bar g =& -\Big(1-\frac{2mr}{\rho^2}\Big) \, dv^2 + 2 \, dvdr + \rho^2 \, d\theta^2 - \frac{4amr\sin^2\theta}{\rho^2} \, d\varphi dv +\\
&- 2a\sin^2\theta \, d\varphi dr + \frac{(r^2+a^2)^2-a^2\Delta\sin^2\theta}{\rho^2}\sin^2\theta \, d\varphi^2,
\end{align*}
given in the so-called Kerr coordinates $\{v,r,\theta,\varphi\}$,
where $m$ and $a$ are two constants called \emph{mass} and \emph{angular momentum}.
We assume that $m$ is positive, and define the quantities $\rho$ and $\Delta$ by 
\begin{align*}
\rho = \sqrt{r^2+a^2\cos^2\theta},
\qquad
\Delta = r^2-2rm+a^2.
\end{align*}
The variables $\{\theta,\varphi\}$ are typical angular coordinates on $S^2$ so that there is a trivial coordinate problem at the axis of symmetry $\theta \rightarrow 0$. The spacetime has a more severe problem at $\rho =0$, where there is a curvature singularity and thus this set must be cut out of the manifold.
For a complete exposition about the Kerr metric and its physical interpretation, one can for example consult \cite{HawkingEllis,Inverno,ONeill(Kerr),Wald}. Let us just mention here that Kerr's metric is of paramount importance in general relativity because it is the unique solution of the vacuum field equations describing an isolated black hole. 

In Kerr's spacetime there is a family of preferred spacelike surfaces defined by constant values of $v$ and $r$. Let $\S$ be one of these surfaces, then $\S$ is compact and topologically $S^2$ ---unless $r=0$. The vector fields $\partial_{\theta}$ and $\partial_{\phi}$ are tangent to $\S$ at every point of $\S$,
the first fundamental form reads (with constant $r$)
$$
g=\rho^2 \, d\theta^2 +\frac{(r^2+a^2)^2-a^2\Delta\sin^2\theta}{\rho^2}\sin^2\theta \, d\varphi^2,
$$
and the following vector fields constitute a frame on $\normsp$ :
\begin{align*}
\xi
&= \frac{1}{\rho^2}\Big( (r^2+a^2) \, \partial_v + \Delta\partial_r + a \, \partial_\varphi \Big), \\
\eta
&= \frac{1}{\rho^2} \Big( a^2\sin^2\theta \, \partial_v + (r^2+a^2) \,\partial_r + a \, \partial_\varphi \Big).
\end{align*}
Notice that $\xi^{\flat}=dr$ and $\eta^{\flat}=dv$. In the basis $\{\partial_\theta,\partial_\varphi\}$, their corresponding shape operators are
\begin{align*}
A_\xi
&= \frac{\Delta}{\rho^2} M_1 \\
A_\eta
&= \frac{1}{\rho^2} \left( (r^2+a^2) M_1 -2\frac{m}{\rho^2} r a^3\sin^3\theta \cos\theta M_2 \right),
\end{align*}
where
\begin{align*}
M_1
=\left( \begin{array}{cc}
\displaystyle{\frac{r}{\rho^2} } &  0 \\
0  & \displaystyle{ \frac{\rho^2\left( r+\frac{m}{\rho^4}a^2(a^2\cos^2\theta-r^2)\sin^2\theta \right)}{(r^2+a^2)^2-a^2\Delta\sin^2\theta}}
\end{array} \right)
\end{align*}
and
\begin{align*}
M_2
=\left( \begin{array}{cc}
0  & \displaystyle{ \frac{1}{\rho^2}} \\
\displaystyle{\frac{\rho^2}{((r^2+a^2)^2-a^2\Delta\sin^2\theta)\sin^2\theta} }& 0
\end{array} \right).
\end{align*}

By Corollary \ref{corollary_n=2} $\S$ is umbilical with respect to a normal direction if and only if $[A_\xi,A_\eta]=0$. Explicitly:
\begin{align}\label{Kerr_umbilicalcondition}
-2 \Delta \frac{m}{\rho^6} r a^3\sin^3\theta \cos\theta [M_1,M_2] = 0.
\end{align}
Equation (\ref{Kerr_umbilicalcondition}) is satisfied if one of the following conditions holds: $[M_1,M_2]=0$ (which is equivalent to $4mr^2 + \rho^2(r-m)=0$ and implies $r<m$), or $\theta\in \{0,\frac{\pi}{2},\pi\}$, or $r=0$, or $a=0$ or $\Delta=0$. The case $a=0$ leads to the Schwarzschild spacetime, describing a non-rotating black hole, and every such $\S$ is actually totally umbilical. This is to be expected, because in this case the spacetime is spherically symmetric, and all round spheres are then totally umbilical \cite{S2007}. Letting this special case aside, among the previous conditions, $\Delta=0$ and $r=0$ are the only possibilities that lead to surfaces $\S$ \emph{entirely} umbilical along a normal direction. We consider them in turn:

\begin{itemize}
\item For the case $r=0$ with any constant $v$, we must keep in mind that the equator in any of these surfaces lies on the spacetime singularity ($\rho=0$ if $\theta=\pi/2$), and thus they have to be avoided. Thus, these surfaces are defined by non-compact, hemi-spherical caps, with either $\theta \in [0,\pi/2)$ or $\theta\in (\pi/2,\pi]$. In any of these options, it is easily seen from the above that $A_\xi = A_\eta =M_1/\cos^2\theta$ with
$$
M_1 =\left( \begin{array}{cc}
0 &  0 \\
0  & \displaystyle{ \frac{m}{a^2}\tan^2\theta }
\end{array} \right).
$$
Thus we deduce that $A_{\xi-\eta}=0$. Note that $\xi -\eta = \partial_v$ on these surfaces, and that $g(\xi,\partial_v)=0$ too. Hence, every such surface is $H$-subgeodesic, and thus also ortho-umbilical, with $H\in$ span$\{\xi\}$. One can further check that these hemi-spherical caps are locally flat.
\item
Suppose now that $\Delta=0$. This requires $m^2 \geq a^2$ and the hypersurface $\Delta =0$ has two connected components given by $r=r_\pm$ with
$$
r_\pm :=  m\pm \sqrt{m^2-a^2} 
$$
except in the case $m=|a|$, called the extreme case, where both of them coincide. 
It follows from the formulas above that $A_\xi=0$ at $r=r_+$ or $r=r_-$. Then $\tildeA_\xi =0$ and $\theta_\xi =\mbox{tr} (A_\xi)= g(\xi,H)=0$ too there, the former saying that any surface with constant $v$ in $\Delta =0$  is umbilical along the normal vector field $\xi$ and the second that $\xi$ is proportional to $\hodge H$. Thus, every such surface is ortho-umbilical.

The total shear tensor for any of these surfaces is given by (with $\rho^2_\pm = r_\pm^2 +a^2 \cos^2\theta$) 
\begin{align*}
\tildeh(\partial_{\theta},\partial_{\theta})
&= a^2\sin^2\theta \, \frac{4mr_\pm^2 + \rho_\pm^2(r_\pm-m)}{2(r_\pm^2+a^2)^2} \, \xi \\
\tildeh(\partial_{\theta},\partial_{\phi})
&= - a^2\sin^2\theta \, \frac{2mr_\pm a\sin\theta\cos\theta}{\rho_\pm^2(r_\pm^2+a^2)} \, \xi \\
\tildeh(\partial_{\phi},\partial_{\phi})
&= - a^2\sin^2\theta \, \frac{\sin^2\theta\left( 4mr_\pm^2 + \rho^2(r_\pm-m) \right)}{2\rho_\pm^4} \, \xi.
\end{align*}
Its image is spanned by $\xi$ and hence, by Theorem \ref{maintheorem} (iii) and Corollary \ref{corollary_uniqueness}, these surfaces are umbilical with respect to $\hodge \xi$. But $\hodge \xi \in$ span$\{H\}$ so that all these surfaces are also pseudo-umbilical. As they are not  totally umbilical, $\xi$ and $\hodge \xi$ (equivalently $H$ and $\hodge H$) must be proportional. This is indeed the case since $\xi$ (equivalently $H$) is null at $\Delta =0$.
\end{itemize}

We have proven the following result:
\begin{proposition}
In the Kerr spacetime with $a\neq 0$, the only surfaces defined by constant values of $v$ and $r$ which are umbilical along a normal direction are those sitting on either 
\begin{enumerate}
\item the (timelike) hypersurface $r=0$ or 
\item the (null) hypersurface $\Delta=0$ ---these exist only when $m\geq |a|$.
\end{enumerate}
In case (1) they are locally flat, non-compact topological disks,  which are ortho-umbilical and $H$-subgeodesic.
In case (2) the surfaces are compact topological spheres both pseudo- and ortho-umbilical. They happen to have a non-vanishing null mean curvature vector field $H$, and thus they are also marginally trapped.
\end{proposition}

The surfaces we have found in case (2), those characterized by constant values of $v$ and by $r=r_\pm$, foliate the null hypersurface defined by $\Delta=0$. In gravitational physics the two connected components of $\Delta=0$ are called the \emph{event horizon} ($r=r_+$) and the \emph{Cauchy horizon} ($r=r_-$), and they enclose the black hole region of the Kerr spacetime \cite{HawkingEllis,Inverno,Kriele,Wald}: a region containing closed trapped surfaces. In the present section we have thus proven that the horizons of the Kerr black hole are foliated by marginally trapped surfaces which are both pseudo- and ortho-umbilical. This fact is already well known in gravitational physics, where they say that the null hypersurface $\Delta =0$ is \emph{expansion- and shear-free} along its null generator. 

\end{document}